\documentclass[12pt,oneside,a4paper]{amsart}

\usepackage{amscd}
\usepackage{amssymb}
\usepackage{a4wide}
\usepackage{amstext}
\usepackage{amsthm}
\usepackage{mathrsfs}
\usepackage{color}

\newtheorem{prop}{Proposition}[section]
\newtheorem{thm}[prop]{Theorem}
\newtheorem{lem}[prop]{Lemma}
 \newtheorem{coro}[prop]{Corollary}

\newtheorem{deft}[prop]{Definition}
\newtheorem{exe}[prop]{Example}
\newtheorem{rem}[prop]{Remark}
\newcommand{\D}{\mathbb{D}}

\newcommand{\dst}{{\displaystyle}}

\newcommand{\rea}{\operatorname{Re}}
\newcommand{\ima}{\operatorname{Im}}

\newcommand\CC{\mathbb{C}}

\newcommand\TT{\mathbb{T}}
\newcommand\DD{\mathbb{D}}
\newcommand\NN{\mathbb{N}}
\newcommand\RR{\mathbb{R}}
\newcommand\wkl{\widetilde{k}_\lambda}
\newcommand{\cp}{\mathbb{C^+}}
\newcommand{\vep}{\varepsilon}

\DeclareMathOperator{\Lin}{Lin}

\DeclareMathOperator{\ctg}{ctg}

\begin{document}
\title[\sf Geometry of reproducing kernels in model spaces near the boundary]
{Geometry of reproducing kernels in model spaces near the boundary}
\author[Baranov, Hartmann, Kellay]{A. Baranov,  A. Hartmann, K. Kellay}
\address{A. Baranov\\ Department of Mathematics and Mechanics\\ St. Petersburg State University\\
St. Petersburg\\ Russia}
\email{anton.d.baranov@gmail.com}

\address{A. Hartmann \& K. Kellay \\IMB\\Universit\'e Bordeaux I\\
351 cours de la Liberation\\33405 Talence \\France}
\email{Andreas.Hartmann@math.u-bordeaux1.fr}
\email{kkellay@math.u-bordeaux1.fr}

\thanks{The work was supported by the Russian Science Foundation grant no 14-41-00010.}

\begin{abstract}
We study two geometric properties of reproducing kernels in model spaces $K_\theta$
where $\theta$ is an inner function in the disc: overcompleteness 
and existence of uniformly minimal
systems of reproducing kernels which do not contain Riesz basic sequences. 
Both of these properties are related to the notion of the Ahern--Clark point. It is shown that
``uniformly minimal non-Riesz"$ $ sequences of reproducing kernels
exist near each Ahern--Clark point which is not an analyticity point for $\theta$, while
overcompleteness may occur only near the Ahern--Clark points of infinite order
and is equivalent to a ``zero localization property". In this context the notion of
quasi-analyticity appears naturally, and as a by-product of our results we give conditions in the
spirit of Ahern--Clark for the restriction of a model space to a radius to be a class of
quasi-analyticity.
\end{abstract}
\maketitle

\section{Introduction and Main Results}

Let $H^2=H^2(\DD)$ denote the standard Hardy space in the unit disk $\DD$, and let 
$\theta$ be an inner function in $\DD$. The {\it model} (or {\it star-invariant}) 
{\it subspace} $K_\theta$ of $H^2$ is then defined as
$$
K_\theta=H^2\ominus \theta H^2.
$$
According to the famous Beurling theorem, any closed subspace of $H^2$ 
invariant with respect to the backward
shift in $H^2$ is of the form $K_\theta$. 
For the numerous applications of model spaces in operator theory 
and in operator-related complex analysis see \cite{nik1, nik2}.

Recall that the function
$$
k_\lambda(z) =
k^\theta_\lambda(z) =
\frac{1-\overline{\theta(\lambda)}\theta(z)}{1-\overline \lambda z}
$$
is the {\it reproducing kernel} for the space $K_\theta$
corresponding to the point $\lambda\in\mathbb{D}$, that is, $(f, k^\theta_\lambda) = f(\lambda)$ 
for any function $f\in K_\theta$.  
We usually omit the index $\theta$
when it is clear from the context which model space we consider.
 In what follows
we denote by $\wkl$ the normalized reproducing kernel, that is, $\wkl = 
k_{\lambda}/\|k_{\lambda}\|_2$.

Geometric properties of systems of reproducing kernels in model spaces
is a deep and important subject which is studied extensively, see 
\cite{fr2, bar06, mp, mp1, bb, bbb, bbb1} for the study of completeness
and \cite{hnp, fr1, bar05, bms, bl, bdhk} for the results about bases 
of reproducing kernels.
The main reason for that is that the geometric properties of reproducing kernels
in a Hilbert space of analytic functions are related to the 
intrinsic analytic properties 
of the space. Let us mention several of such connections:
\begin{itemize}
\item
a system of reproducing kernels $\{k_\lambda\}_{\lambda\in \Lambda}$ is complete
in $K_\theta$ if and only if $\Lambda$ is a uniqueness set for $K_\theta$, i.e.,
if $f\in K_\theta$ vanishes on $\Lambda$, then $f=0$;
\item 
$\{\wkl\}_{\lambda\in \Lambda}$ is a Riesz basic sequence if and only if 
$\Lambda$ is an {\it interpolating sequence}, i.e., for every data 
$a_\lambda$ such that $\sum_{\lambda\in \Lambda} |a_\lambda|^2 \|k_\lambda\|^{-2} <\infty$,
there exists a solution $f\in K_\theta$ of the interpolation 
problem $f(\lambda) = a_\lambda$, $\lambda\in\Lambda$;
\item
$\{\wkl\}_{\lambda\in \Lambda}$ is  a Riesz basis if and only if 
$\Lambda$ is a {\it complete  interpolating sequence}, i.e., the above interpolation 
problem has a unique solution.
\end{itemize}

Another motivation relies on the fact that  systems of eigenfunctions 
of certain second order differential equations 
are canonically unitarily equivalent 
to systems of reproducing kernels in model spaces (see \cite{hnp} or \cite{mp}).

We are interested in the following two problems. The first of them, posed 
by Nikolai Nikolski,
is related to the overcompleteness phenomenon in  model spaces. Recall that a system of vectors 
$\{x_n\}$  in a separable Banach space $X$ is said to be {\it overcomplete} if every subsequence
$\{x_{n_k}\}$ is complete in $X$. Such sequences have for instance been discussed by
Szeg\H{o} who showed that the sequence $\frac{\dst 1}{\dst t+\lambda_n}$ is overcomplete in
$C([0,1])$ whenever $\lambda_n\to\infty$ (see \cite{CH}). 
Klee \cite{Kl} has shown that every separable Banach space 
possesses overcomplete sequences. 
For a system of reproducing kernels $\{k_{\lambda_n}\}$ 
in a reproducing kernel Hilbert space of functions analytic in some domain $\Omega$,
a trivial reason for being overcomplete is that $\lambda_n$ tends to some point 
$\lambda_0\in\Omega$. In most of the classical spaces (e.g., the Hardy or the Bergman space
in the disc) there are no "nontrivial"$ $ overcomplete  systems of reproducing kernels.
However, in model spaces such systems may exist.
\medskip
\\
{\bf Problem 1.} {\it Describe those inner functions $\theta$ for which
there exists $\Lambda=\{\lambda_n\}\subset \DD$ such 
that $\{k_\lambda\}_{\lambda\in \Lambda}$ is overcomplete in $K_\theta$, but 
$\Lambda$ has no accumulation points in the domain of analyticity of  the elements 
of $K_\theta$.  }
\medskip

This problem was explicitly addressed by Chalendar, Fricain and Partington
in \cite{CFP} (we refer also to \cite{FrT}). In \cite{CFP} it was shown that 
the overcompleteness for reproducing kernels in model spaces
is related to the notion of the Ahern--Clark point for $K_\theta$.
Recall that for an inner function $\theta$, the Ahern--Clark set is defined by
$$
AC(\theta) =\Big\{\zeta\in \TT \text{ : } \sum_{j}
\frac{1-|z_j|^2}{|\zeta-z_j|^2}+\int_\TT
\frac{d\nu(z)}{|\zeta-z|^2}<\infty\Big\},
$$
where $z_j$ are  the zeros of $\theta$ (counting multiplicities) and $\nu$ is the singular
measure generating its singular factor. 

Recall that $AC(\theta)$ is exactly the set where $\theta$ has a unimodular
non-tangential boundary value and a finite angular derivative. 
Such points are also referred to as Julia or Carath\'eodory points.
Ahern and Clark showed in \cite{AC} that a function $f\in K_\theta$ 
has a finite non-tangential limit at a point $\zeta\in \TT$ if and only if 
$\zeta\in AC(\theta)$. In this case the function $k_\zeta$ belongs to $K_\theta$
and is the reproducing kernel at the boundary point $\zeta$.
Note also that if we denote by $\sigma(\theta)$ 
the {\it boundary spectrum of $\theta$}: $\sigma(\theta)
=\{\zeta\in\TT:\liminf_{z\to \zeta}|\theta(z)|=0\}$, 
then $\theta$ admits analytic continuation through $\TT\setminus\sigma(\theta)$. 
Thus, clearly,  $\TT\setminus \sigma(\theta) \subset AC(\theta)$, however the 
most interesting situation for us is when $\zeta\in \sigma(\theta) \cap AC(\theta)$.

Concerning the overcompleteness problem, it is shown in 
\cite{CFP} that if $\{k_{\lambda}\}_{\lambda\in \Lambda}$
is overcomplete, then (assuming $AC(\theta)\neq \TT$)
$$
\text{dist}\,(\Lambda, \TT\setminus AC(\theta)) >0.
$$ 
For the proof see Corollary \ref{rep5}.

One of our main results says that the overcompleteness is equivalent 
to the {\it localization property} introduced recently by Abakumov, Belov
and the first author in \cite{ABB} in the context of de Branges spaces (i.e., 
essentially, in the case of model spaces $K_\theta$ such that
$\sigma(\theta)$ consists of one point). Recall that the Stolz angle 
$\Gamma_\gamma$, $\gamma>1$, at the point $\zeta\in \TT$ is defined as 
$\Gamma_\gamma(\zeta) = \{ z\in \DD: |z-\zeta|\le \gamma(1-|z|)\}$.
                                                   
\begin{deft} Let $\zeta\in \sigma(\theta)$. We say that the space $K_\theta$ 
has localization property at the point $\zeta$, if any nonzero $f\in K_\theta$ 
has only finitely many zeros in any Stolz angle at $\zeta$.  
\end{deft}

In this definition, Stolz angle can be replaced by any 
region of the form $\{z\in \DD: |z-\zeta|^N \leq \gamma (1-|z|)\}$ 
for some $N, \gamma>0$ (see Lemma \ref{rep}).
It should be observed that in our more general setting 
the different conditions of localization given in \cite{ABB} are no longer equivalent. The very
term of localization introduced in \cite{ABB}
means that the zeros of the functions in the space (the de Branges spaces they consider)
are localized in certain regions. For our definition we
pick the condition which, on the contrary, claims the existence of almost zero free regions (almost 
meaning up to a finite number). Still, our definition says that the zeros are localized outside a
Stolz angle (or more general domains, see Lemma \ref{rep}).

We can now state our first result.

\begin{thm}
\label{loc}
Let $\zeta\in \TT$. 
Then
the following statements are equivalent:

\begin{enumerate}
\item 
There exists a sequence $\lambda_n\to \zeta$ such that 
$\{k_{\lambda_n}\}$ is overcomplete. 

\item $K_{\theta}$ has the localization property at $\zeta$. 
\end{enumerate}
Moreover, if these conditions are satisfied, then $\zeta\in AC(\theta)$. 
\end{thm}

The most interesting case of this theorem is when $\zeta\in \sigma(\theta)$ since
outside the spectrum of $\theta$ every function $f\in K_{\theta}$ has analytic continuation
which immediately gives localization and overcompleteness. 

Theorem \ref{loc} solves Problem 1. However, the points with localization
do not admit an explicit description. In some special situations such descriptions are
given in \cite{ABB} where the relations of the localization property in de Branges spaces 
with the structure of its subspaces and with the spectral theory of canonical systems is revealed.

Now we state two corollaries of Theorem \ref{loc}. 
The first result provides a necessary condition for overcompleteness 
and extends significantly the results of \cite{CFP}. The Ahern--Clark set of higher order
$AC_n(\theta)$ is defined analogously to $AC(\theta)$
(see Section 2) and is related 
to the existence of non-tangential boundary values of derivatives in $K_\theta$ \cite{AC}.

\begin{thm}
\label{high}
If $\lambda_n\to \zeta\in \TT$ and $\{k_{\lambda_n}\}$ is overcomplete 
then $\zeta\in \bigcap_{n=0}^{\infty}AC_n(\theta)$.
\end{thm}

The converse of this result is not true as will be shown by the example given in 
Section \ref{ovloc} of 
a point $\zeta \in \bigcap_{n=0}^{\infty}AC_n(\theta)$ which is not 
a point of localization (and, thus, there is no overcomplete 
sequence $\{k_{\lambda_n}\}$ with $\lambda_n\to \zeta$).

Next we introduce the notion of strong localization which will turn out to be 
a sufficient condition for overcompleteness 
at a point of the boundary spectrum of $\theta$ (it seems that no such examples
were given in \cite{CFP}).  
Recall that with each $\alpha\in \TT$
we can associate a singular measure $\sigma_{\alpha}$ on $\TT$,
the so-called Clark measure (see Section 2 for details). The condition 
$\zeta\in AC(\theta)$ is equivalent to the fact that 
$\sigma_\alpha(\{\zeta\}) \ne 0$ for some (unique) $\alpha\in \TT$.

\begin{deft}
Let $\zeta\in \bigcap_{n\geq 0} AC_n(\theta)$.
We say that $K_\theta$ has strong localization at $\zeta$, 
if the system $\{(z-\zeta)^{-k} \}_{k\ge 1}$ is complete in $L^2(\sigma_\alpha)$ 
for some \textup(any\textup) $\alpha\in \TT$ such that $\sigma_\alpha(\{\zeta\})=0$.
\end{deft}

Again we should emphasize that we are in a more general situation than in \cite{ABB}. In that
paper strong localization gave a more precise information on the 
localization of the zeros (namely, each zero except a finite number was located
near a point mass of the Clark measure $\sigma_\alpha$). 
As it turns out in de Branges spaces considered in \cite{ABB}
this is equivalent to the density definition
given above.
A priori, in our setting it is not immediately clear why strong localization
should imply localization. That this is actually the case will be discussed below in Corollary 
\ref{strongloc}.

Interestingly, the above definition, which may look rather abstract at first sight,
can be connected with another well known property, namely that of quasi-analyticity.

\begin{thm}\label{quasi}
The point $\zeta = 1$ is a point of strong localization for $K_{\theta}$
if and only if $K_{\theta}|_{[0,1]}$ is a class of quasi-analyticity.
\end{thm}

Since quasi-analyticity is a stronger requirement than just being $C^{\infty}$-smooth,
this theorem leads naturally to the question whether it is possible to characterize 
quasi-analyticity of $K_{\theta}|_{[0,1]}$ in the spirit of Ahern--Clark, i.e.\ in terms of the
behavior of the zeros and the singular measure of $\theta$ near $\zeta$.

Using some classical results on polynomials approximation, it is possible to give the
the following sufficient condition in terms of the Clark measure 
which is in the spirit of
another way of characterizing Ahern--Clark points of arbitrary order (cf. \eqref{bab11}
below). 

\begin{thm}
\label{dens}
Let $K_\theta$ be a model space in the disc and let $\sigma$ 
be some Clark measure for $K_\theta$. Assume that, for some $\vep>0$, 
\begin{equation}
\label{expo}
\int_{\TT} \exp\Big(\frac{\vep}{|\eta - \zeta|}\Big) d\sigma(\eta) <\infty.
\end{equation}
Then $\zeta$ is a strong localization point for $K_\theta$, i.e., 
$K_{\theta}|_{[0,\zeta]}$ is a class of quasi-analyticity.
\end{thm}

Unfortunately, and contrarily to the Ahern--Clark situation, 
it cannot be expected that a condition of type 
\eqref{expo} with the
exponential replaced by some appropriate function is necessary and sufficient for strong localization 
or quasi-analyticity. We will discuss this matter more thoroughly
through the results of Borichev and Sodin \cite{BS} after the proof of Theorem \ref{dens} in Section
\ref{S4}. 
\\

With Theorem \ref{quasi} in mind, we 
are able to deduce the following consequence.

\begin{coro}\label{strongloc}
Strong localization at $\zeta \in\TT$ implies localization at $\zeta$.
\end{coro}

This corollary will automatically provide a sufficient condition for localization at $\zeta\in\TT$
and thus existence of overcomplete systems accumulating to $\zeta$.
Examples are given in \cite{ABB} showing that there are points of
localization which are not of strong localization.

\medskip

Let us now turn to the second problem considered in the present paper.
\medskip
\\
{\bf Problem 2.} {\it Describe those inner functions  
for which there exists a uniformly minimal sequence 
$\{\tilde{k}_\lambda\}_{\lambda\in \Lambda}$ which does not contain any Riesz sequence. }
\medskip

Such systems of reproducing kernels will be called UMNR
systems (i.e., uniformly minimal non-Riesz). Note that again most of the classical spaces 
of analytic functions do not possess such systems. Also there exist model spaces
for which the class UMNR is empty. This is for instance the case for $K_\theta$ with
$\theta(z) = \exp\big(a\frac{1+z}{1-z}\big)$, $a>0$, or for the corresponding 
model space in the upper half-plane associated with the inner function $\Theta(z) 
= \exp(iaz)$. Indeed, in this case it is known that any incomplete system of 
normalized reproducing kernels contains  Riesz sequences.

It turns out that UMNR systems of reproducing kernels in $K_\theta$ exist  
if and only if  $\theta$ has "nontrivial"$ $ Ahern--Clark points: 
$\sigma(\theta) \cap AC(\theta) \ne \emptyset$.

\begin{thm}
\label{umnr}
Let $\zeta\in AC(\theta)\cap\sigma(\theta)$. Then there exists
a sequence $\lambda_n\to \zeta$ such that $\{\widetilde{k}_{\lambda_n}\}$ is UMNR.
\end{thm} 

Note that
an overcomplete system is never uniformly minimal. Still, an overcomplete system can be
minimal. A related result concerns the possibility of extracting a uniformly minimal system from a
minimal system. The following observation will follow from our discussions 
in a rather simple way.

\begin{thm}
\label{MNUM}
Let $\zeta\in AC(\theta)$ which is not a point of localization.
Let $\{z_n\}$ be a sequence in $\D$. Then
\begin{enumerate}
\item[(1)] If $z_n\to \zeta\in\TT$ non-tangentially then $\{\tilde{k}_{z_n}\}$ is a minimal
sequence that does not contain any uniformly minimal sequence.
\item [(2)] If $z_n\to \zeta$, but $\|k_{z_n}-k_{\zeta}\|\not\to 0$, then
$\{\tilde{k}_{z_n}\}$ contains a uniformly minimal sequence.
\end{enumerate}
\end{thm}

Note that from Theorems \ref{loc} and \ref{high} it is easily seen that it 
$\zeta\in AC(\theta)\setminus AC_1(\theta)$ then we are in the setting of 
the above theorem.
\bigskip

A final word concerning notation. In this paper the notation $U(z)\lesssim V(z)$  
means that there is a constant $C>0$
such that $U(z)\leq CV(z)$ holds for all suitable values of the variable 
$z$. We write $U(z) \asymp V(z)$ if $U(z)\lesssim V(z)$
and $V(z)\lesssim U(z)$.
\bigskip

\section{Preliminaries} 

\subsection{Necessity of the Ahern--Clark condition}
We start with some simple observations on the geometry of vectors in Hilbert spaces.
We write $\displaystyle x_n\xrightarrow{\mathrm{w}} x_0$ if the sequence $x_n$ converges weakly 
to $x_0$ in a Hilbert space $H$.

\begin{lem}
\label{rep4}
A normalized sequence $\{x_n\}$ in a Hilbert space contains a Riesz sequence 
if and only if it contains a subsequence $\{x_{n_k}\}$
such that $\displaystyle x_{n_k} \xrightarrow{\mathrm{w}}0$.
\end{lem}
 
\begin{proof} Clearly, if  $\{x_{n_k}\}$ is a Riesz sequence, then
$\displaystyle x_{n_k} \xrightarrow{\mathrm{w}}0$. Conversely, 
if $\{x_n\}$ contains a subsequence weakly converging to zero, then, 
proceeding inductively, we can choose a subsequence $\{x_{n_k}\}$ such that
\begin{equation}\label{gram}
 \sum_{\ell}\sum_{k\neq \ell}|(x_{n_k},x_{n_{\ell}})|^2<1.
\end{equation}
Then, denoting by $G=(x_{n_k},x_{n_{\ell}})$ the Gram matrix associated with $\{x_{n_k}\}$, 
and writing $G=Id+G_0$ we see that \eqref{gram} implies 
that $G$ is bounded and invertible, whence
$x_{n_k}$ is a Riesz sequence (see \cite[Volume 2, p.171]{nik2}).
\end{proof}

In the next corollary we will use the following fact:
for any inner function $\theta$ the set $K_\theta\cap C(\overline{\DD})$
is dense in $K_\theta$. While this is trivial when $\theta$ is a Blaschke product,
it is in general a nontrivial fact due to Aleksandrov \cite{al8}.

\begin{coro}
\label{supcond}
The sequence of normalized reproducing kernels 
$\{\widetilde{k}_\lambda\}_{\lambda\in \Lambda}$ in $K_\theta$
contains a Riesz 
subsequence if and only if $\sup_{\lambda\in \Lambda}\|k_{\lambda}\|=\infty$.  
\end{coro}      

\begin{proof} 
Without loss of generality, let $\|k_{\lambda_n}\|\to \infty$, $n\to\infty$.
For any $f\in K_\theta\cap C(\overline{\DD})$, 
$$
(f,\widetilde{k}_{\lambda_n})=\frac{f(\lambda_n)}{\|k_{\lambda_n}\|}\to 0.
$$
Hence, $\displaystyle \widetilde{k}_{\lambda_n} \xrightarrow{\mathrm{w}}0$, 
and it suffices to apply Lemma \ref{rep4}. The converse statement is immediate.
\end{proof}

\begin{coro}                            
\label{rep5}
If $\lambda_n\to \zeta\in \TT$  and $\{\widetilde{k}_{\lambda_n}\}$ 
is overcomplete or UMNR, then $\sup_n \|k_{\lambda_n}\|<\infty$, 
whence $\zeta\in AC(\theta)$.
\end{coro}

\begin{proof} 
In both situations $\{\widetilde{k}_{\lambda_n}\}$ does not contain any Riesz sequence, so
that by the preceding corollary we have $\sup_{n}\|k_{\lambda_n}\|<\infty$.
Recall that  $\|k_{\lambda_n}\|^2 = \frac{1-|\theta(\lambda_n)|^2}{1-|\lambda_n|^2}$, whence
$$
\limsup_{n\to\infty} \frac{1-|\theta(\lambda_n)|^2}{1-|\lambda_n|^2} <\infty.
$$
Now  the classical Julia--Carath\'eodory theorem implies that $\zeta\in AC(\theta)$.
\end{proof}
\medskip


\subsection{Higher order Ahern--Clark condition and Clark measures}
\label{haw}
Let $z_j$ be the zeros of an inner function $\theta$ (counting multiplicities) and 
let $\nu$ be the singular measure generating its singular factor. 
We say that $\zeta\in \TT$ is in $AC_n(\theta)$, 
the Ahern--Clark set of order $n$, if 
\begin{eqnarray}\label{hoAC}
\sum_{j}\frac{1-|z_j|^2}{|\zeta-z_j|^{2n+2}}+
\int_\TT\frac{d\nu(z)}{|\zeta-z|^{2n+2}}<\infty.
\end{eqnarray}
By the results of Ahern--Clark, $\zeta\in AC_n(\theta)$ 
if and only if there exist non-tangential limits of $f^{(k)}(z)$, 
$0\le k\le n$, as $z\to\zeta$, for every $f\in K_{\theta}$.  
Note that in this notation $AC(\theta)= AC_0(\theta)$.

Recall that the measure $\sigma_\alpha$, $\alpha\in \TT$,
from the representation
$$
\frac{\alpha+\theta(z)}{\alpha-\theta(z)}=
\int_\TT\frac{1+\bar\xi z}{1-\bar\xi z}d\sigma_\alpha(\xi)
$$
is called the Clark measure for $K_\theta$ (see \cite{cl}). 
We sometimes write $\sigma_\alpha^\theta$ to emphasize the
dependence on $\theta$. Any function $f\in K_\theta$ 
has non-tangential boundary values $\sigma_\alpha$-everywhere \cite{polt},
$\|f\| = \|f\|_{L^2(\sigma_\alpha)}$,  and the map 
\begin{equation}
\label{clark-oper}
V:h\in L^2(\sigma_\alpha) \mapsto 
f(z)=(\alpha-\theta(z))\int_\TT\frac{h(\xi)}{1-\bar\xi z} d\sigma_\alpha(\xi)
\end{equation}
is a unitary map from $L^2(\sigma_\alpha)$ onto $K_\theta$.
If $\zeta\in AC(\theta)$, then there exists $\alpha_0$ such that 
$\sigma_{\alpha_0}(\{\zeta\})>0$.  It is well known (see, e.g., \cite[VII-2]{sar}) 
that 
\begin{equation}
\label{bab11}
\zeta\in AC_n(\theta)\iff \int_\TT
\frac{d\sigma_{\alpha}(\eta)}{|1-\bar\eta \zeta|^{2n+2}}<\infty,
\qquad  \alpha\neq\alpha_0.
\end{equation}
\medskip


\subsection{Transfer to the upper half-plane}
\label{sictrans}
In what follows it will be often convenient to pass to an equivalent problem
in the half-plane setting where the estimates and computations become much simpler.
For $\zeta\in \TT$, consider the conformal mapping 
\begin{equation}
\label{bab10}
w(z) = i\frac{\zeta +z}{\zeta - z}, 
\end{equation}
which maps $\DD$ onto the upper half-plane $\cp =\{z\in \CC: \ima z>0\}$, 
the unit circle $\TT\setminus\{\zeta\}$ onto the real axis $\RR$,
and the point $\zeta$ to $\infty$.
For an inner function $\theta$ in $\DD$, put $\Theta(w) = 
\theta\big(\zeta\frac{w-i}{w+i}\big)$. Then $\Theta$ 
is an inner function in $\cp$.

It is well known that
$$
T: f\to \frac{1}{w+i}f\Big(\zeta \frac{w-i}{w+i}\Big)
$$
is a unitary mapping from $H^2(\DD)$ to the Hardy space $H^2(\cp)$ 
in the upper half-plane and $TK_\theta = K_\Theta
= H^2(\cp) \ominus \Theta H^2(\cp)$ (see \cite[Chapter A6]{nik2}). 

The following property of the spaces $K_\Theta$ will often be used in what follows: 
given $f\in H^2(\cp)$, 
\begin{equation}
\label{bab20}
f\in K_\Theta \iff \overline{f(t)}\Theta(t) \in H^2(\cp),
\end{equation}
which means that the function $\overline{f(t)}\Theta(t)$ on $\RR$ 
coincides with the non-tangential boundary values 
of some function in $H^2(\cp)$.

Let $\nu_0$ be a measure on $\TT$. Then the change of variable  
$$
d\nu_0(\tau)  = \frac{d\nu(t)}{t^2+1}, \qquad t\in \RR, \quad 
\tau = \zeta\frac{t-i}{t+i} \in \TT,
$$
gives us a measure $\nu$ on $\RR$. The Ahern--Clark conditions 
of order $n$ for the point infinity may then be rewritten in terms of the zeros 
$w_j = x_j+iy_j$ of $\Theta$
and of the corresponding singular measure $\nu$ as follows:
\begin{equation}
\label{nomer}
\infty \in AC_n(\Theta) \iff 
\sum_{j} y_j (1+|w_j|^2)^n +\int_\RR (1+x^2)^n d\nu(x) <\infty,
\end{equation}
while in terms of the Clark measures $\sigma_\alpha$ for $\Theta$ (note that we 
use the same notation) the Ahern--Clark condition of order $n$ becomes
\begin{equation}\label{ACCM}
\infty \in AC_n(\Theta)\iff \int_{\RR} (1+x^2)^n d\sigma_\alpha <\infty
\end{equation}
for all $\alpha\in\TT$,  $\alpha \ne \alpha_0 = \lim_{y\to+\infty}\Theta(iy)$.
In particular, the usual Ahern--Clark condition (of order 0) means that 
$\sigma_\alpha(\RR) <\infty$.

Now we state the localization and strong localization properties at 
the point $\infty$.

\begin{deft} The space $K_\Theta$ in $\cp$  has localization property at the point $\infty$, 
if any nonzero $f\in K_\Theta$ has only finitely many zeros in any Stolz angle 
$\Gamma_\gamma = 
\{z\in \cp:|z|>1 \text{ and } \ima z \ge \gamma |\rea z|\}$, $\gamma>0$. 

The space $K_\Theta$ has the strong localization at $\infty$ 
if for any Clark measure $\sigma_\alpha$ except 
$\alpha =\lim_{y\to \infty} \Theta(iy)$ the 
polynomials belong to the space $L^2(\sigma_\alpha)$ and are dense there.
\end{deft}

Note that in this definition we have added the condition $|z|>1$ in order to distinguish 
$\Gamma_{\zeta}$ from 
the Stolz angle at 0.

Both of the above definitions are equivalent to the localization 
(strong localization) property
at $\zeta$ for the function $\theta$ related to the function $\Theta$
by  \eqref{bab10}.

In the half-plane setting it is easy to see,
using an idea from \cite{ABB}, that in the definition 
of the localization at $\infty$ the Stolz angle may be replaced by any domain 
of the form $\Gamma_{\gamma, \beta} = 
\{\ima z > \gamma |\rea z|^{\beta}, |z|>1\}$ where $\gamma>0$, $\beta\in \RR$.

\begin{lem}
\label{rep}
If $K_\Theta$ has the localization property at $\infty$, then any nonzero $f\in K_\Theta$
has only a finite number of zeros in any domain $\Gamma_{\gamma, \beta}$.
\end{lem}
 
\begin{proof}
Assume the converse and let $f\in K_\Theta$ have infinitely many 
zeros in some domain $\Gamma_{\gamma, \beta}$.
Choose a subsequence $\{\lambda_n\}$ of such zeros 
such that $|\lambda_{n+1}| > 2|\lambda_n|$. Then the infinite product
$G(z) = \prod_n (1-z/\lambda_n)$  converges and $\lim_{|x|\to\infty}
|x|^{-N} |G(x)| = \infty$ for any $N>0$ (the limit is taken over $x\in \RR$). 
Here we use the fact that 
$$
 |1-x/\lambda_n| \ge |x-\lambda_n|/|\lambda_n|
 \ge |\ima\lambda_n|/|\lambda_n|
 \ge \gamma |\lambda_n|^{-|\beta|-1},
$$
and the lacunarity of $\{\lambda_n\}$.

Now we may choose a sequence $iy_n$ which is so sparse that the infinite product
$\widetilde{G}(z) = \prod_n (1-z/iy_n)$  converges and 
$|\widetilde{G}(x)| \le C|G(x)|$ on $\RR$ for some $C>0$ 
(e.g., take $y_n = \lambda_{10n}$). Then 
$g(z) = \widetilde{G}(z)f(z)/G(z)$ is in $H^2(\cp)$ and also
$$
\overline{g(t)}\Theta(t) = \overline{f(t)}\Theta(t) 
\widetilde{G}^*(t)/G^*(t)  \in H^2(\cp),
$$
where $G^*(z) = \overline{G(\overline{z})}$.
So, by \eqref{bab20}, $g\in K_\Theta$ and $g(iy_n) = 0$, a contradiction to the 
localization at $\infty$.
\end{proof}                                                                  
\bigskip                                    


\section{Overcompleteness and localization}
\label{ovloc}

In this section we give the proofs of Theorems \ref{loc} and \ref{high}. 
For this we need one more equivalent form of localization.

\begin{prop}
\label{rep1}
If $\zeta\in \TT$ is not a point of localization for $K_\theta$, then 
for any sequence $\lambda_n\to \zeta$ there exist a subsequence 
$\lambda_{n_k}$ and $f\in K_\theta$, $f\neq 0$, such that $f(\lambda_{n_k})=0$.
\end{prop}

The converse is trivially true. Observe from Lemma \ref{rep} that localization is only determined by the 
behavior of zeros inside Stolz domains or their generalized form $\Gamma_{\gamma,
\beta}$.

\begin{proof} Pass to $\CC^+$ 
by the conformal mapping \eqref{bab10} 
which maps $\zeta$ to $\infty$. The condition 
that there is a function with infinitely many points on the radius 
means now that there exists $f\in K_\theta$ and $y_n\to +\infty$
such that $f(iy_n)=0$. 
Let $\{\lambda_n\}$ be any sequence tending to infinity. Let us choose
a lacunary product $E=\prod(1-z/iy_n)$. We can always choose an even more 
lacunary product  $G=\prod(1-z/\lambda_{n_k})$ with 
$\lambda_{n_k}\in \{\lambda_n\}$ such that $|G(x)|\leq |E(x)|$ on $\RR$.
Then, making use of \eqref{bab20}, it is easy to see that
$\widetilde{f}(z)=f(z) G(z)/E(z)$ will belong to $K_\theta$ 
and vanish on $\{\lambda_{n_k}\}$.
\end{proof}
\smallskip
\noindent
{\it  Proof of Theorem \ref{loc}}. $(2)\Longrightarrow (1)$ is trivial, any 
$\lambda_n$ which tends to $\zeta$ along the radius gives an overcomplete system.

$(1)\Longrightarrow(2)$ follows from Proposition \ref{rep1}.
\qed
\bigskip
\\
{\it Proof of Theorem \ref{high}}.
Assume that there exists an overcomplete system $\{k_{\lambda_n}\}$
with $\lambda_n\to \zeta$, but $\zeta\notin AC_n(\theta)$ and $\zeta\in AC_{n-1}(\theta)$ 
for some $n\ge 1$. Note that we already know from Corollary \ref{rep5} that necessarily
$\zeta\in AC(\theta)=AC_0(\theta)$.
By \eqref{bab11}, there exists $\alpha\in \TT$  
such that
$$
\int_\TT\frac{d\sigma_\alpha(\eta)}{|1-\bar\eta\zeta|^{2n}}=\infty\quad{and}
\quad \int_\TT\frac{d\sigma_\alpha(\eta)}{|1-\bar\eta\zeta|^{2k}}<\infty, \quad k<n.
$$
Passing to $\CC^+$ by the conformal mapping \eqref{bab10},
we get a space $K_\Theta$ in $\CC^+$ with a Clark measure $\mu
= \sigma_\alpha^\Theta$ such that 
$$
\int_\RR |t|^{2n}d\mu(t)=\infty\quad\text{and}\quad \int_\RR 
|t|^{2k}d\mu(t)<\infty,\qquad k<n
$$
(see the discussion in Subsection \ref{sictrans}).
Consider the measure $d\tilde\mu(t)= |t|^{2n}d\mu(t)$. 
We thus have $\tilde\mu(\RR)=\infty$, but $\displaystyle 
\int_\RR \frac{d\tilde \mu (t)}{t^2+1}<\infty$.
Define an inner function $\widetilde{\Theta}$ in $\cp$ by the formula 
\begin{equation}
\label{tild}
i\frac{1+\widetilde{\Theta}(z)}{1-\widetilde{\Theta}(z)} = 
\int\bigg(\frac{1}{t-z}-\frac{t}{t^2+1} 
\bigg)d\tilde \mu(t).
\end{equation}
Then, clearly, $\tilde \mu  = \sigma_1^{\widetilde{\Theta}}$, 
the Clark measure for $K_{\widetilde{\Theta}}$. 

Note that the model space $K_{\widetilde{\Theta}}$ has no 
localization at $\infty$. Indeed, if $1\ne \lim_{y\to+\infty}\widetilde{\Theta}(iy)$,
then $\infty \notin AC(\widetilde{\Theta})$ by \eqref{ACCM}, since 
$\tilde\mu(\RR) = \infty$.
If $1 = \lim_{y\to+\infty}\widetilde{\Theta}(iy)$ 
and $\infty \in AC(\widetilde{\Theta})$, then by definition of the angular derivative
at $\infty$ we must have $0<\lim_{y\to\infty} y(1-\tilde{\Theta}(iy))<\infty$ so that
$$
\lim_{y\to\infty} \frac{1}{y}
\bigg|\frac{1+\widetilde{\Theta}(iy)}{1-\widetilde{\Theta}(iy)}\bigg| >0.
$$

However, it follows from \eqref{tild} that the above limit 
is zero. We conclude that $\infty$ is not an Ahern--Clark point
for $\widetilde{\Theta}$ and, thus, not a localization
point for $K_{\widetilde{\Theta}}$.

We will now use the unitary operator $V_+:L^2(\tilde{\mu})\to K_{\tilde{\Theta}}$
already mentioned earlier, which in the
half-plane setting is defined by 
$V_+f(z)=(1-\tilde{\Theta}(z))\int_\RR\frac{u(t)}{t-z}d\tilde \mu(t) $. By
Theorem \ref{loc}, there exists $u\in L^2(\tilde \mu)$ such that 
the function $h \in K_{\widetilde{\Theta}}$ defined by 
$$
h(z)=(1-\widetilde{\Theta}(z)) 
\int_\RR\frac{u(t)}{t-z}d\tilde \mu(t)  =
(1-\widetilde{\Theta}(z)) 
\int_\RR\frac{u(t)t^{2n}}{t-z}d\mu(t) 
$$
has infinitely many (simple) zeros of the form $\{iy_m\}$, $y_m\to\infty$.  
Moreover, note that the functions $\varphi_m(z):=\frac{h(z)}{z-iy_m}$ belong to 
$K_{\widetilde{\Theta}}$, vanish at $iy_{\ell}$, $\ell\neq m$, 
and are linearly independent. Write $\varphi_m(z)=V_+u_m$.
Clearly for an appropriate finite linear combination
$v$ of $u_m$, we achieve 
$$
\int_\RR v(t) t^{k} d\mu(t)=0,\qquad k\le 2n-1.
$$
By construction, the function
$$
g(z) = \int_\RR\frac{v(t)t^{2n}}{t-z}d\mu(t) 
$$
vanishes at $iy_m$ for $m$ sufficiently big (i.e., $m\ge m_0$).

Now let 
$$
f(z)=\int_\RR\frac{v(t)}{t-z}d\mu(t).
$$
Clearly, $v\in L^2(\mu)$ and so 
$(\alpha - \Theta) f\in K_\Theta$ (recall that $\mu
= \sigma_\alpha^\Theta$). Let us show that $f(iy_m)=0$ for $m$ sufficiently big. Indeed,
using $1= t^{2n} z^{-2n}-(t-z)\sum_{k=0}^{2n-1} t^k z^{-k-1}$, 
we can write
$$
f(z)
=-\sum_{k=0}^{2n-1}\frac{1}{z^{k+1}}\underbrace{\int_\RR v(t)t^kd\mu(t)}_{0}+
\frac{1}{z^{2n}}\underbrace{\int_\RR\frac{v(t)t^{2n}}{t-z}d\mu(t)}_{g(z)}.
$$
Hence, $f(iy_m) = 0$, $m > m_0$, which contradicts the fact that
$\infty$ is a localization point for $K_\Theta$.
\qed
\medskip
\begin{exe}
{\rm The converse is not true: there exist points 
$\zeta \in \bigcap_{n=0}^{\infty}AC_n(\theta)$ which are not 
points of localization for $K_\theta$. In view of the conformal mapping, 
it is sufficient to construct a Blaschke product $B$ in $\CC^+$ such that 
$\infty \in \bigcap_{n=0}^{\infty}AC_n(B)$ but $\infty$ is not a localization 
point for $K_B$. 

Let $B$ be the Blaschke product with zeros
$$
 z_n = |n|^\alpha  {\rm sign}\, n +i \exp(-|n|^{1/\beta}), 
\qquad n\in\mathbb{Z},
$$
where $1<\alpha<\beta$. Put $E(z) = \prod_n (1-z/\overline{z_n})$. 
It is then clear that $B = \gamma E^*/E$ for some unimodular constant $\gamma$
(recall that we define $g^*(z) = \overline{g(\overline{z})}$). 
By \eqref{nomer} we have $\infty \in \bigcap_{n=0}^{\infty}AC_n(B)$.

By standard estimates of canonical products 
 (see, e.g.,  \cite[Ch. 2]{lev}) we have 
$$
\log\bigg|\frac{E(z)}{{\rm dist}\, (z, \{\overline{z_n}\})}\bigg| 
\asymp |z|^{1/\alpha}, \qquad |z|>1,
$$
and in particular, $\log|E(x)|\asymp |x|^{1/\alpha}$, $x\in\mathbb{R}$, $|x|\ge 1$. 
Let $F(z)$ be an entire function of order less than $1/\alpha$
with imaginary zeros, say, 
$F(z) = \prod_n (1-z/(2^n i))$. Clearly, $F/E \in L^2(\mathbb{R})$
and, hence, $f = F/E \in H^2(\cp)$, since any entire function of order less than 1 is of 
Smirnov class in the upper half-plane. Also, as in the proof of Lemma \ref{rep},
$$
\overline{f(t)}B(t) = \frac{\overline{F(t)}}{\overline{E(t)}}\cdot 
\frac{\overline{E(t)}}{E(t)} = 
\frac{F^*(t)}{E(t)}, \qquad t\in\mathbb{R}.
$$
By similar reasons as above, $F^*/E \in H^2(\cp)$ whence, by \eqref{bab20}, 
$f\in K_B$. Since $f$
has infinitely many imaginary zeros, we conclude that $\infty$
is not a localization point for $K_B$. }
\end{exe}   
\bigskip


\section{Strong localization and quasi-analyticity}\label{S4}
\noindent
{\it Proof of Theorem \ref{quasi}.}
Observe first that strong localization requires by definition that $1\in \bigcap_{n\ge 0}AC_n(\theta)$, and
if $K_{\theta}|_{[0,1]}$ is a class of quasi-analyticity then all the derivatives of $f\in K_{\theta}$ 
are supposed to exist
radially so that $1\in \bigcap_{n\ge 0}AC_n(\theta)$, 
and we can implicitly admit this condition.

Recall that if $1\in \bigcap_{n\ge 0}AC_n(\theta)$ then $(z-1)^{-n}\in L^2(\sigma)$
for every $n\in \NN$  (see for instance \cite[VII-2]{sar}).
Necessarily in this case $\sigma(\{1\})=0$. 
Again we will use the fact that $K_{\theta}=V L^2(\sigma)$, where $\sigma$ is the Clark measure that we 
suppose associated with $\alpha=1$, and the isometry $V$ is 
defined by \eqref{clark-oper}.

Suppose $1$ is a point of strong localization.
Pick an arbitrary function $f=Vh\in K_{\theta}$, and suppose
that $f^{(n)}(1)=0$ for every $n\in\NN^*$. In order to show that $K_{\theta}|_{[0,1]}$ is a class of 
quasi-analyticity we have to check that $f$ vanishes identically.

Since $f$ has a zero of arbitrary order at $1$, the function
$g$ defined by $g(z)=f(z)/(1-\theta(z))$ has also a zero of arbitrary order at 1 (note that $\theta$
has the same regularity at $1$ as any function in $K_{\theta}$, and $\lim_{r\to 1}\theta(r)\neq 1$).
So $\lim_{r\to 1}g^{(n)}(r)
=g^{(n)}(1)=0$ for every $n\in\NN$. Clearly
\[
 g^{(n)}(z)=\frac{d^n}{dz^n}\int_{\TT}\frac{h(\zeta)}{1-\overline{\zeta}z}d\sigma(\zeta)
 =n!\int_{\TT}\frac{\overline{\zeta}^nh(\zeta)}{(1-\overline{\zeta}z)^{n+1}}d\sigma(\zeta).
\]
Observe that 
\[
 \left|\frac{\overline{\zeta}^nh(\zeta)}{(1-\overline{\zeta}z)^{n+1}}\right|
 \lesssim\frac{|h(\zeta)|}{|1-\overline{\zeta}|^{n+1}}.
\]
The function on the right hand side is integrable since $h\in L^2(\sigma)$ and
$(1-z)^{-k}\in L^2(\sigma)$ for every $k$. Since we also have pointwise convergence, by
Lebesgues' dominated convergence theorem we conclude
\[
 0=\lim_{r\to 1} g^{(n)}(r)=n!\int_{\TT}\frac{\overline{\zeta}^nh(\zeta)}
  {(1-\overline{\zeta})^{n+1}}d\sigma(\zeta).
\]
It remains to use an inductive argument. For $n=0$, we conclude that $h\perp (1-\zeta)^{-1}$
(with respect to the scalar product in $L^2(\sigma)$). Suppose 
$h\perp (1-\zeta)^{-k}$ for $1\le k\le n$. Note that
\[
 \frac{\overline{\zeta}^n}
  {(1-\overline{\zeta})^{n+1}}=\frac{1}
  {(1-\overline{\zeta})^{n+1}}-\frac{(1-\overline{\zeta}^n)}
  {(1-\overline{\zeta})^{n+1}}
 =\frac{1}
  {(1-\overline{\zeta})^{n+1}}-\frac{(1+\overline{\zeta}+\cdots+\overline{\zeta}^{n-1})}
  {(1-\overline{\zeta})^{n}}.
\]
Since $\frac{\dst (1+\overline{\zeta}+\cdots+\overline{\zeta}^{n-1})}
{\dst (1-\overline{\zeta})^{n}}$ is in the space generated by $(1-\zeta)^{-k}$, 
$1\le k\le n$, integrating against $h$ in the last term with respect to $d\sigma$ yields 0. Hence
\[
 \int_{\TT}\frac{h(\zeta)}
  {(1-\overline{\zeta})^{n+1}}d\sigma(\zeta)=0,
\]
which achieves the induction.
We have thus proved that if the function $f$ vanishes to arbitrary order at $1$, then 
$h\perp (1-\zeta)^{-n}$ for every $n\in \NN^*$. By strong localization, these functions generate
the whole space $L^2(\sigma)$, so that $h=0$, and hence $f=0$.
\medskip

For the converse, the argument is almost the same.
Suppose $K_{\theta}|_{[0,1]}$ is a class of quasi-analyticity.
Pick any $h\in L^2(\sigma)$ and suppose $h\perp (1-\zeta)^{-n}$, $n\in \NN^*$.
By construction $f=Vh\in K_{\theta}$, and, associating with this $f$ the function $g$ as
above, we notice that 
\[
 \lim_{r\to 1} g^{(n)}(r)=n!\int_{\TT}\frac{\overline{\zeta}^nh(\zeta)}
  {(1-\overline{\zeta})^{n+1}}d\sigma(\zeta)=0
\]
(again observe that ${\zeta}^n/  (1-{\zeta})^{n+1}$ is in the space generated by
$(1-\zeta)^{-k}$, $1\le k\le n+1$). Thus $f=(1-\theta)g$ has zero of arbitrary order at $1$, in other
words $f^{(n)}(1)=0$ for every $n\in \NN$. By quasi-analyticity $f$ has to vanish on $[0,1]$ and
thus on $\DD$, which implies that $h=0$. We conclude that $(1-\zeta)^n$, $n\in\NN^*$,
generates a dense subspace.
\qed
\bigskip
\\
{\it Proof of Corollary \ref{strongloc}.}
We still suppose $\zeta=1$ for simplicity.
By Theorem \ref{quasi}, strong localization is equivalent to quasi-analyticity.

Recall also that we can again assume $1\in \bigcap_{n\ge 0} AC_n(\theta)$.

Now suppose there is a function $f\in K_{\theta}$ with infinitely many zeros $z_k$  
in a Stolz angle at $1$. Then in particular 
$\lim_{z\stackrel{\angle}{\longrightarrow} 1}f(z)=
\lim_{k\to +\infty}f(z_k)=0$. Then also
$\lim_{z\stackrel{\angle}{\longrightarrow}  1}\frac{f(z)-f(1)}{z-1}
=\lim_{k\to +\infty}\frac{f(z_k)}{z_k}=0$. By induction
we obtain that $f^{(n)}(1)=0$ for every $n\in \NN$. Since $K_{\theta}|_{[0,1]}$ is
quasi-analytic, we conclude that $f$ vanishes identically.
\qed
\bigskip
\\
{\it Proof of Theorem \ref{dens}.}
Passing to an equivalent problem in the space $K_\Theta$ in the upper half-plane
(related to $K_\theta$ by \eqref{bab10}) and the point $\infty$,
we obtain a Clark measure $\mu = \sigma^\Theta_\alpha$ for $K_\Theta$ such that 
\begin{eqnarray}\label{condfre}
\int_\RR e^{\vep |t|} d\mu(t) <\infty.
\end{eqnarray}
As we have seen before, strong localization in the the upper half is related with
weighted polynomial approximation which is one of the most classical subjects 
of analysis (for a detailed survey see \cite{fre, koo}).
It is well known that under condition \eqref{condfre} the polynomials are dense in $L^2(\mu)$
(see, e.g., \cite[Theorem II.5.2]{fre}, or \cite[Exercise A4.8.3(c)]{nik2}), and so $\infty$ is a strong localization point
for $K_\Theta$. The value $\alpha$ is not exceptional for $K_\Theta$ since
$\sigma_\alpha^\theta$ has no point mass at $\zeta$.
\qed

\medskip

As already mentioned in the introduction, and contrarily to the Clark measure formulation
\eqref{bab11} or \eqref{ACCM} of the Ahern--Clark condition for
existence of non-tangential higher order derivatives at boundary points, a condition of type
$\int_{\RR}\Phi(t)d\mu(t)<\infty$ (case of the line) cannot give a necessary and sufficient condition
for completeness of polynomials, and hence quasi-analyticity. We will discuss this through the
results of \cite{BS} as presented in \cite[Exercise A4.8.3($\ell$)]{nik2}.

In order to do so, consider the sequence $\Lambda_{\rho}=\{n^{1/\rho}:n=1,2,\ldots\}$, $\rho>0$, and
the weight $w_{m,s}(\lambda)=\lambda^se^{-c\lambda^m}$, $m>1$, $c>1$, $s\in\RR$.
Set $\mu=\sum_{\lambda\in\Lambda_{\rho}}w_{m,s}^p(\lambda)\delta_{\lambda}$.
This singular measure is finite and it is possible (after a possible normalization) to associate with it
a model space $K_{\Theta}$ (we will consider the case $p=2$ here).

According to \cite{BS},
if $m\ge 1/2$, then the polynomials are always dense in 
$$L^p(\mu)=\ell^p(\Lambda_{\rho},
w_{m,s}^p)=\{x=(x(\lambda))_{\lambda\in \Lambda_{\rho}} :\sum_{\lambda\in\Lambda_{\rho}}
|x(\lambda)w_{m,s}(\lambda)|^p <\infty\}
$$
(again, we are only interested in the case $p=2$ here). However
$$
 \int_{\RR}\Phi(t)d\mu(t)=\sum_{n\ge 1}\Phi(n^{1/\rho})\frac{n^{ps/\rho}}{e^{cpn^{m/\rho}}}
$$
converges if $\Phi(x)=O(e^{cpx^{m'}})$ for $m'<m$ and diverges if
$\liminf_{x\to\infty}\Phi(x)e^{-cpx^{m'}}>0$ for $m'\ge m$ (and $s>0$). So, integrability against
a function $\Phi$ cannot be necessary and sufficient.

Considering the case $0<\rho=m<1/2$, there exists a constant $c_0=\pi\ctg(\pi\rho)$ such that
if $c>c_0$, the polynomials are dense, and if $c>c_0$ they are not (there are also some discussions
on the case $c=c_0$; see \cite{BS} or \cite[Exercise A4.8.3($\ell$)]{nik2} for all these
results). In this situation the integrability of (the sub-exponential function) 
$\Phi(t)=e^{c_0px^{\rho}}$ against
$d\mu$ thus gives a hint at quasi-analyticity or not. Still, the function $\Phi$ heavily depends on
$\rho$ and thus on the space $K_{\Theta}$. So there is no universal function characterizing
quasi-analyticity in terms of the Clark measure as is the case for $n$-th order derivatives
given in \eqref{ACCM}.
\bigskip


\section{UMNR sequences of reproducing kernels}

\begin{lem} 
\label{rep2}
If a normalized sequence $\{x_n\}$ is uniformly minimal and 
contains no Riesz sequences, then $\{x_n\}$ contains a subsequence $\{x_{n_k}\}$ 
such that 
\begin{enumerate}
\item[(i)] $\displaystyle x_{n_k}\xrightarrow{\mathrm{w}}x$,
\item[(ii)] $\displaystyle x_{n_k}-x$ is a Riesz sequence,
\item[(iii)] $x\notin\overline{\Lin}\{x_{n_k}-x\}$.
\end{enumerate}
Conversely, any such $\{x_{n_k}\}$ is UMNR.
\end{lem}

\begin{proof} 
We start with the sufficient condition. Since $\{x_n\}$ is uniformly bounded, we 
can pick a weakly convergent subsequence $\{x_{n_k}\}$
of $\{x_n\}$ (which obviously is UMNR). By Lemma \ref{rep4}, 
$\displaystyle x_{n_k}\xrightarrow{\mathrm{w}}x\neq 0$.
Since $\{x_{n_k}\}$ is uniformly minimal, no subsequence can converge in norm, so that
we can assume $0<\varepsilon \le \|x_{n_k}-x\|\le M$ and hence $\{x_{n_k}-x\}$ can be
supposed normalized and $\displaystyle x_{n_k}-x\xrightarrow{\mathrm{w}}0$. Again by
Lemma \ref{rep4}, and passing possibly to a subsequence we may assume that 
$\{x_{n_k}-x\}$ is a Riesz sequence. 

It remains to check (iii). Since 
$\{x_{n_k}\}$ is uniformly minimal, there exists a biorthogonal
system $\{y_l\}$ such that $\sup_l\|y_l\|<\infty$. Let $ z_k=x_{n_k}-x$, which was shown
to be a
Riesz sequence. Then 
$$
(z_k,y_l)+(x,y_l)=(z_k+x,y_l)=\delta_{k,l}
$$ 
and, for fixed $l$ and since $\{z_k\}$ is a Riesz sequence, we have $(z_k,y_l)\to 0$ 
as $k\to \infty$. So $(x,y_l)=0$ and hence $\{y_l\}$ is biorthogonal to $\{z_k\}$.
Let $H_0=\overline{\Lin}\{z_k\}$ and  $y_l=y_l'+y_l''$ where $y_l'\in H_0$ 
and $y_l'' \in H_0^\perp$. Then $\{y_l'\}$ is biorthogonal to $\{z_k\}$ in $H_0$ 
and
$$
(x,y_l)=0\iff (x,y_l')+(x,y_l'')=0.
$$
If $x\in H_0$, then $(x,y_l'')=0$ and so $(x,y_l')=0$. However,  
$\{y_l'\}$ is the biorthogonal of a Riesz basis in $H_0$ and thus is a Riesz basis itself in $H_0$.
Whence $x=0$ in contradiction to our hypothesis on $x$. 
Thus $x\notin H_0$ which shows (iii).

Conversely, suppose $\{x_{n_k}\}$ satisfies (i)--(iii). In particular, by (iii) we have
$x\notin H_0$. In the same notation as introduced before, since $\{z_k\}$ is a Riesz basis in $H_0$,
its biorthogonal $\{y_l'\}$ is also a Riesz basis in $H_0$.
Clearly we can always 
find $y_l''$ with bounded norms to get $(x,y_l'')=-(x,y_l')$, and so 
the vectors $y_l=y_l'+y_l''$ have uniformly bounded norms
and form a biorthogonal system to $\{z_k+x\}$. Hence $\{x_{n_k}\}$ is uniformly minimal. Note that 
$x\neq 0$ (remember that $x\notin H_0$),
and so, (i) and Lemma 2.1 imply that $\{x_{n_k}\}$ cannot contain
any Riesz sequence.
\end{proof}
\medskip

We state an immediate consequence of the above lemma which we will use in the proof of
Theorem \ref{umnr}.

\begin{coro}\label{coroUMNR}
If $\{x_n\}$ is a normalized sequence tending weakly to $x\neq 0$ which has a 
subsequence not converging in norm to $x$, then $\{x_n\}$ contains a UMNR sequence.
\end{coro}

\begin{proof}
In view of the hypotheses, we can suppose $0<\varepsilon\le
\|x_{n_k}-x\|\le M<+\infty$ for some suitable subsequence. 
Also $x_{n_k}-x\xrightarrow{\mathrm{w}}0$, and by Lemma \ref{rep4}, passing to a subsequence, we
can suppose that $\{x_{n_k}-x\}$ is a Riesz sequence. This allows us to claim that 
if $x\in H_0$ then we can always pass to 
a subsequence generating a subspace not containing $x$. It remains to apply Lemma \ref{rep2}
to conclude.
\end{proof}

\medskip
\noindent
{\it Proof of Theorem \ref{umnr}}.
Since $\zeta\in\sigma(\theta)$, there exists a sequence $(z_n)_n\subset\DD$ converging to
$\zeta$  such that $\theta(z_n)\to 0$, $n\to\infty$.
In particular 
\[
 \|k_{z_n}\|^2=\frac{1-|\theta(z_n)|^2}{1-|z_n|^2}\asymp\frac{1}{1-|z_n|^2} \to \infty,
 \quad n\to\infty.
\]
(note that in view of the Ahern--Clark condition,
the sequence $(z_n)_n$ has to tend tangentially to $\zeta$).  
On the other hand, when $\lambda\to \zeta$ non-tangentially, then, since $\zeta\in AC(\theta)$,
$k_\lambda\to k_\zeta$ in $K_\theta$, in particular $\|k_\lambda\|\to \|k_\zeta\|$. 
Thus  we may choose a 
sequence $\lambda_n$ (on suitable intervals connecting $z_n$ to some fixed Stolz angle
at $\zeta$) such that $\lambda_n\to \zeta$, but 
$$
\|k_{\lambda_n}\|=2\|k_\zeta\|.
$$
Let us show that  $k_{\lambda_n}\xrightarrow{\mathrm{w}}k_\zeta$. Indeed, 
for $g\in K_\theta\cap \overline{C(\DD)}$ (which, as already mentioned earlier, 
is a dense subset of $K_\theta$,
see \cite{al8}), 
$$
(g,k_{\lambda_n})=g(\lambda_n)\to g(\zeta)=(g,k_\zeta).
$$
Since the norms $\|k_{\lambda_n}\|$ are bounded, by the Banach--Steinhaus theorem, 
$k_{\lambda_n}\xrightarrow{\mathrm{w}}k_\zeta$. Thus 
$$
\widetilde{k}_{\lambda_n}=\frac{k_{\lambda_n}}{\|k_{\lambda_n}\|} 
\xrightarrow{\mathrm{w}}\frac{k_\zeta}{2\|k_{\zeta}\|}=\frac{\widetilde{k}_{\zeta}}{2}.
$$
and in particular $\widetilde{k}_{\lambda_n}$ has no subsequence converging in norm to 
$\widetilde{k}_{\zeta}$.
By Corollary \ref{coroUMNR}, $\{\widetilde{k}_{\lambda_n}\}$ is UMNR.
\qed

\begin{rem}
{\rm Note that Theorem \ref{umnr} provides a description of those $\lambda_n\to \zeta$ 
for which $\widetilde{k}_{\lambda_n}$ is UMNR (or contains such system). 
All we need is that $\sup\|k_{\lambda_n}\|<\infty$ and 
$\|k_{\lambda_n}\|\not\to \|k_\zeta\|$.  }
\end{rem}

Before discussing explicit examples of UMNR sequences $\{\lambda_n\}$, we briefly discuss the
proof of Theorem \ref{MNUM}.

\medskip 

\noindent
{\it Proof of Theorem \ref{MNUM}}.
By Corollary \ref{supcond} we can suppose that $\sup_{n}\|k_{z_n}\|<\infty$.

Consider (2).  By the above remark, since $\|k_{\lambda_n}\|\not\to \|k_\zeta\|$, we deduce
that $\widetilde{k}_{\lambda_n}$ contains a UMNR sequence and in particular 
a uniformly minimal sequence.

Consider (1). Since $\zeta$ is not a point of localization, there exists an infinite sequence
$\{z_n\}$ and a non-vanishing function $f\in K_{\theta}$, such that $f(z_n)=0$, $n\in\NN$.
By the backward shift invariance, we can assume that the zeros of $f$ are simple.
Then the sequence
$\{\widetilde{k}_{z_n}\}$ is minimal. Indeed, the sequence $\{\varphi_n\}$ defined by
$\varphi_n(z)=f(z)/(z-z_n)$ gives a biorthogonal system. On the other hand
$\{\widetilde{k}_{z_n}\}$ cannot be uniformly minimal. Indeed, since we are in an Ahern--Clark point,
we have $k_{z_n}\to k_{\zeta}$, and hence the distance $\|k_{z_n}-k_{z_{n+1}}\|$ goes to
zero while $\|k_{z_n}\|$ is uniformly bounded, contradicting thus uniform minimality.
\qed

\begin{exe} {\rm We give an example of a UMNR system of normalized 
reproducing kernels $\{\tilde k_{\lambda_n}\}$ having 
an additional property: the system $\{k_{\lambda_n}\}$ is complete in $K_\Theta$. 
To simplify the estimates we will construct an example in the half-plane setting.

Let $z_n = x_n +i y_n$, $n\ge 1$, be a sequence in $\cp$ such that
$x_n>0$, $x_{n+1}>x_n +1$ and $\sup_n \sum_{k\ne n} |x_n-x_k|^{-1} <\infty$.
Furthermore, let $0< s_n < 1$, put $t_n =x_n + s_n$ for $n\ge 2$,
and assume that 
\begin{equation}
\label{bab12}
\frac{y_n}{s_n^2} \asymp \frac{1}{t_n^2}, 
\qquad
\sum_{k\ne n} \frac{y_k}{(t_n-x_k)^2} \lesssim \frac{1}{t_n^2},
\qquad
\sup_n \sum_{k\ne n}\frac{s_k}{|t_k-t_n|}<\infty.
\end{equation}
Clearly, taking sufficiently small $y_n$ and defining $s_n = x_n\sqrt{y_n}$
we can achieve all the properties.

We will show that under the above assumptions
$\{\tilde k_{t_n}\}_{n\ge 2}$ is a complete UMNR system
in $K_\Theta$ where $\Theta$ is the Blaschke product with zeros $z_n$.

Define entire functions $E$ and $G$ as zero genus canonical products 
with zeros $\{\overline{z}_n\}_{n \ge 1}$ and $\{t_n\}_{n\ge 2}$ 
respectively. Then standard estimates
of canonical products (combined with the last inequality in \eqref{bab12}) 
show that for $z\in \CC$ such that 
$|z-t_n| = {\rm dist}\, (z, \{t_k\})$ we have
\begin{equation}
\label{bab15}
\bigg|\frac{G(z)}{E(z)}\bigg| \asymp \frac{|z-t_n|}{|z-x_n +iy_n|}\cdot\frac{1}{|z|+1}. 
\end{equation}
Indeed, if $|z-t_n| = {\rm dist}\, (z, \{t_k\})$,
then $\sum_{k\ne n} |z-\overline z_k|^{-1} \le C$ 
for some constant $C$ independent on $z$ and $n$, whence
$$
\sum_{k\ne n} \log \bigg|\frac{1-z/t_k}{1-z/\overline {z}_k}\bigg| = 
\sum_{k\ne n} \log 
\bigg|1+ \frac{\overline z_k - t_k}{ z- \overline{z}_k} \bigg| + O(1) = O(1).
$$
In particular, it follows from \eqref{bab15}, that
\begin{equation}
\label{bab16}
\bigg|\frac{G'(t_n)}{E(t_n)}\bigg| \asymp\frac{1}{s_n t_n}.
\end{equation}
Note that $\Theta(z):=\overline{E(\overline{z})}/E(z)$ is a Blaschke product in $\cp$
with zeros $z_n$. Moreover, the class $E \cdot K_\Theta$ consists of entire functions
and coincides with the so-called de Branges space $\mathcal{H}(E)$
(see \cite{br}). 

Note that, by \eqref{bab12} and a straightforward estimate,
$$
2\pi \|k_{t_n}\|^2  = |\Theta'(t_n)| \le 2\sum_k \frac{y_k}{(t_n-x_k)^2}  
= \frac{2y_n}{s_n^2} + 2\sum_{k\ne n} \frac{y_k}{(t_n-x_k)^2} \lesssim \frac{1}{t_n^2}
 \to 0,\quad n\to\infty.
$$
It follows that $\{\tilde k_{t_n}\}$ does not contain any subsequence weakly 
converging to zero. Indeed, taking $f(z) = (z-\overline{z}_1)^{-1} \in K_\Theta$ 
we see that $|(f,\tilde{k}_{t_n})|=|f(t_n)|/\|k_{t_n}\|\gtrsim t_n^2\to\infty,
n\to\infty$.
Thus, by Lemma \ref{rep4}, $\{\tilde k_{t_n}\}$ does not contain Riesz subsequences. 
In fact, in the upper half-plane case 
the condition  $\sup_\lambda |\lambda| \cdot \|k_\lambda\|$ 
is necessary and sufficient for $\{\tilde k_\lambda\}$ to contain a Riesz subsequence
(compare with Corollary \ref{supcond}).

Let us verify that $\{\tilde k_{t_n}\}$ is uniformly minimal. It is easy to see
that the biorthogonal system to $\{\tilde k_{t_n}\}_{n\ge 2}$ is given by 
$$
g_n(z) = \frac{E(t_n)\|k_{t_n}\|}{G'(t_n)} \cdot \frac{G(z)}{E(z)(z-t_n)}.
$$ 
We need to show that $\sup_n \|g_n\|<\infty$. 
Let $I_1 = \big[0, \frac{x_1 + x_2}{2}\big]$ 
and $I_k = \big[\frac{x_{k-1} + x_k}{2}, \frac{x_{k} + x_{k+1}}{2}\big]$, $k>1$.
Making use of \eqref{bab15} and \eqref{bab16}, we see that
$$
\begin{aligned}
\|g_n\|^2 & \lesssim \int_{I_n}
\frac{s_n^2 dx}{((x-x_n)^2+y_n^2)(x^2+1)} 
+
 \sum_{k\ne n} \int_{I_k}
\frac{s_n^2 (x-t_k)^2 dx }{((x-x_k)^2+y_k^2)(x-t_n)^2(x^2+1)}
+ O(1) \\
& \lesssim
\frac{s_n^2}{y_nt_n^2} + 
s_n^2 \sum_{k\ne n} \int_{I_k} \frac{(x-x_k)^2 +s_k^2}
{((x-x_k)^2+y_k^2)(x-t_n)^2(x^2+1)}dx 
+ O(1) \\
& \lesssim 
s_n^2 \sum_{k\ne n}   \frac{1}{t_k^2 |t_n-t_k|} 
+
s_n^2 \sum_{k\ne n}   \frac{s_k^2}{y_k t_k^2 |t_n-t_k|^2}  +O(1) = O(1).
\end{aligned}
$$
In the last inequality we used the last condition in 
\eqref{bab12} and the fact that $s_n<1$.

Analogous estimates show that $G/E \notin H^2$. Indeed, 
$$
\sum_{k=2}^\infty \int_{I_k} \bigg|\frac{G(t)}{E(t)}\bigg|^2 dt \asymp 
\sum_{k=2}^\infty \int_{I_k} \frac{(x-t_k)^2}{((x-x_k)^2+y_k^2)(x^2+1)}dx 
\asymp 
\sum_{k=2}^\infty \frac{s_k^2}{y_kt_k^2}.
$$
However, the last series diverges by the first condition in \eqref{bab12}.

Finally, we need to show that $\{k_{t_n}\}_{n\ge 2}$ is complete in $K_\Theta$. Assume
that $h\in K_\Theta$ is orthogonal to $\{k_{t_n}\}_{n\ge 2}$, whence $h(t_n) = 0$. Then 
the entire function $H=Eh$ is divisible by $G$, i.e., $H= GS$ for some entire function 
$S$. We have $S = hE/G$ in $\cp$ whence $S$ is in the Smirnov class in $\cp$
(see, e.g., \cite[Part 2, Chapter 1]{hj})
and $|y^{-1}S(iy)| \to 0$, $y\to+\infty$, by \eqref{bab15}. On the other hand,
$$
\overline{S(\overline z)} = \overline{h(\overline z)} \cdot
\frac{\overline{E(\overline z)}}{E(z)} \cdot \frac{E(z)}{\overline{G(\overline z)}},
$$ 
and so $S$ is in the Smirnov class in the lower half-plane and 
$|y^{-1} S(iy)|\to 0$, $y\to -\infty$ (we used the fact that
$h\in K_\Theta$ and so $\overline{h(\overline z)} \cdot
\overline{E(\overline z)}/E(z)\in H^2(\cp)$). By a theorem of M.G.~Krein
\cite[Part II, Chapter 1]{hj} $S$ is of zero exponential type and 
thus the estimates along the imaginary axis imply that 
$S$ is a constant. If $S \ne 0$, then $G/E \in H^2$, a contradiction. }
\end{exe}

\begin{rem}
{\rm In the above example the constructed UMNR system is also complete
in $K_\theta$.
Do such examples exist in the general case? Namely, assume that $\sigma(\theta)$
consist of one point (or of finite number of points). Does there exist a complete 
UMNR system of reproducing kernels? }
\end{rem}

\begin{rem}
{\rm In Example 1 the points were chosen on the real axis. This is not always possible.
E.g., if $z_n = n + i n^{-3/2}$, $n\in \mathbb{N}$, then $\infty$ is an Ahern--Clark 
point for the corresponding Blaschke product $\Theta$, but 
$t^2 |\Theta'(t)| \to\infty$, $t\to \infty$. Hence, any minimal system 
of normalized reproducing kernels $\{\tilde k_{t_n}\}$ contains a Riesz subsequence. }
\end{rem}

\end{document}